\documentclass{article}      
\usepackage{graphicx}
\usepackage{amsthm,amssymb,amsfonts,amsmath}

\usepackage{bbm,color}
\usepackage{hyperref}
 \usepackage{mathptmx}      
\def\call{{\rm c\kern-3.5pt                                    
 \vrule height 5.0pt width 0.4pt depth -0.5pt \phantom {.}}}
\newcommand{\rme}{{\rm e}}
\newcommand{\rmd}{{\rm d}}

\newcommand{\ep}{{\varepsilon}}
\newcommand{\1}{{\bf 1}}
\newcommand{\pd}{{\partial}}
\newcommand{\F}{{\mathcal F}}
\newcommand{\PP}{{\mathbb P}}
\newcommand{\QQ}{{\mathbb Q}}
\newcommand{\RR}{{\mathbb R}}
\newcommand{\ZZ}{{\mathbb Z}}
\newcommand{\EE}{{\mathbb E}}
\newcommand{\HV}{{\rm H\negthickspace V}}

\newtheorem{theo}{Theorem}[section]
\newtheorem{prop}{Proposition}[section]
\newtheorem{rem}{Remark}[section]

\begin{document}
\title{Telegraph Processes with Random Jumps and Complete Market Models
}


\author{Nikita Ratanov\\
              Universidad del Rosario, Cl.\,12c, No.~4-69, Bogot\'a, Colombia \\        
             Email: nratanov@urosario.edu.co       
}


\date{}
\maketitle

\begin{abstract}
We propose a new generalisation of jump-telegraph process with variable velocities and jumps.
Amplitude  of the jumps and velocity values are random, and they  depend 
on the  time spent by the process in the previous state  of  the underlying Markov process.

This construction is applied to markets modelling.  
The distribution densities and the moments satisfy some integral equations of the Volterra type.
We use them for characterisation of the  equivalent risk-neutral measure and for the expression of
historical volatility in various settings. The fundamental equation is derived by similar arguments.

Historical volatilities are computed numerically. 


\end{abstract}

\section{Introduction}
The model of non-interacting particles which move with alternating finite velocities
was first introduced  by \cite{Taylor}.
 Later,   the model was developed by   \cite{Gold}
 in connection with a certain hyperbolic partial differential equation. 
 In 1956 Mark Kac (see \cite{Kac}) began to study  the telegraph model in detail. 
Assuming the random time intervals $T_n$ between the velocity's reversal to be independent 
 and exponentially distributed, $T_n\sim\mathrm{Exp}(\lambda)$,  
Kac derived the telegraph (damped wave) equation
 for the distribution density $p=p(x, t)$  of the particles' positions,
 \[
 \frac{\pd^2 p}{\pd t^2}+2\lambda \frac{\pd  p}{\pd t }=c^2 \frac{\pd^2 p}{\pd x^2}.
 \]
 
Afterwards, the telegraph process
and its many generalisations have been studied in great detail.
 In particular, the generalisations towards motions with the velocities alternated in
 gamma- or Erlang-distributed random instants have been studied many times,
 see e. g. \cite{DiC2,DiC1,Zacks1}. 
 Telegraph processes with random velocities have been  considered by  \cite{Zacks2}.

Applications of telegraph processes 
to market modelling have been presented first by \cite{Kabanov},
and then, by \cite{R1999,Franco}.
Now these applications are  transformed into the theory of 
Markov-modulated market models based on 
 telegraph processes with alternated constant velocities, 
see e. g. \cite{R2007,R2010} and  \cite{ON2012} (see also the survey in \cite{KR}). 
One of the key principles of such a modelling 
 is that the models are based on  \emph{observable} parameters  such
as velocity and  jump amplitude. Replacing the measure, we only change
  the underlying distributions of time intervals between 
 velocity reversals.

In this paper we assume that the telegraph particle  moves with alternated 
\emph{random and variable} 
velocities  performing jumps of 
\emph{random} amplitude whenever the velocity is changed. 
The telegraph processes of this type have been studied 
earlier only under the assumption of mutual independence 
of velocity values and jump amplitudes, see \cite{Zacks2,ON2012, DiC2013-2}. 
Here we assume that the actual velocity regime and subsequent jump  are determined by 
the  functions of the time spent by the particle in the previous state.
We assume also that the time intervals between the state reversals 
have sufficiently arbitrary distributions.
Under these assumptions we obtain the version of telegraph process
which possesses some accelerating/damping properties.
The paper is a continuation of the paper \cite{R13}, where 
such generalisations of the telegraph processes began to be  studied.
Here the problem is considered in a bit more general setting and with financial applications.

Such a model with deterministic velocities and jumps is studied in detail by \cite{R2007,DiC2013}.
Moreover, earlier we proposed the option pricing model 
based on jump-telegraph processes, \cite{R2007}.
In this paper we use these processes with random velocity and jumps 
for the purposes of financial modelling. In particular,
this  corresponds better to the technical 
analysis of oversold and overbought markets. 

If the random jump amplitudes are statistically
 independent of the underlying continuous process, 
 then the market model is typically incomplete
(see the classical paper by  \cite{Merton}, 
  the review by \cite{JDM}
of the jump-diffusion models, 
and also by \cite{ON2012}  for the models based on the telegraph processes).

We profess here the approach of  complete markets.
In \cite{CoxRoss} the market model is based on the simple jump process
and thus with a single source of randomness. Thus the model is complete.
The model, proposed in this paper, typically remains to be complete and arbitrage-free
(similar to another simple model with fixed and deterministic jump amplitude \cite{R2007}). 
In contrast with \cite{R2007}, in our recent setting the closed formulae for  option prices do not exist.
To analyse memory properties of the proposed model
 we numerically evaluate a historical volatility.
 
 The paper is organised as follows.
 The underlying processes are described in Section  \ref{sec:2}.
Section \ref{sec:mart} is devoted to a version of Doob-Meyer decomposition 
which permits to characterise martingales in our version of jump-telegraph processes. 
 The market model (together with the fundamental equation) is 
presented in Section \ref{sec:3}. We focus on the calculations
of historical volatility in Section \ref{sec:num}.  

\section{Generalised jump-telegraph processes}\label{sec:2}
\setcounter{equation}{0}

Let  $(\Omega, \mathcal{F}, \PP)$ be a complete probability space
with given filtration $\mathcal{F}_t,\; t\geq0$ satisfying the usual hypotheses, \cite{Protter}.
We start with a two-state continuous-time Markov process
$\ep=\ep(t)\in\{0, 1\},\; t\geq 0$, adapted to $\mathcal{F}_t.$ 
Fixing  the initial state of   $\ep$,
consider the conditional probabilities $\PP_i,\; i\in\{0, 1\}$ with respect to the initial state of $\ep$,
\[\PP_i(\cdot):=\PP(\cdot~|~\ep(0)=i).\]
The corresponding expectations will be denoted by $\EE_i\{\cdot\}$.
Assume that sample paths of $\ep=\ep(t),$ $t\geq0,$ are  right-continuos a. s. 

To fix the distribution properties of  process $\ep$
 we begin with the set of independent random variables 
$T_n, n\in\ZZ,\; T_n\geq0$ with alternated distributions.
Denote the respective distribution functions by $F_0,$ $F_1,$
the survival functions by $\bar F_0,\; \bar F_1$ and 
the densities by $f_0,\; f_1$. 
The subscript indicates the starting 
position of the alternation which corresponds to the initial
state of $\ep$, that is the distribution of $T_n$ depends on $\ep(0)=i\in\{0, 1\}.$
Precisely,  under probability $\PP_i$
the distribution function of $T_n$ is $F_i$, if $n$ is odd, and $F_{1-i}$, if $n$ is even.

Random variables $T_n$ are the time intervals between 
successive switching of Markov process $\ep$.
Let $\mathfrak{T}=\{\tau_n\}$ be the Markov flow of switching times. Then $T_n=\tau_n-\tau_{n-1}$.
We assume the usual non-explosion condition,
\[
\tau_\infty:=\lim \tau_n=+\infty,\qquad \PP\text{-a.s.}
\]
Moreover, let  $\tau_0=0$, i.e. process $\ep$ starts at a switching instant.

The latter assumption can be neglected. 
If the process is observed beginning from time $s,\;  0=\tau_0<s<\tau_1$,
the corresponding conditional distributions can be described by the survival functions
\begin{equation}\label{barFbarf}
\begin{aligned}
\bar F_i(t~|~s)=&\PP_i(T>t~|~T>s)=\frac{\PP_i(T>t)}{\PP_i(T>s)}=\frac{\bar F_i(t)}{\bar F_i(s)},\\
& 0\leq s<t,\; i=0, 1.\end{aligned}\end{equation}
Therefore, the corresponding densities are 
\[
f_i(t~|~s)=-\frac{\pd}{\pd t}\bar F_i(t~|~s)=\frac{f_i(t)}{\bar F_i(s)},\quad 0\leq s<t,
\quad i=0, 1\]
Here $T=T_1=\tau_1>0$ is the first switching time.

Consider a particle, which moves on $\RR$ 
under alternated velocity regimes $c_0$ and $c_1,$ starting from the origin.
  The velocities are described by  two   piecewise continuous   functions
$c_i=c_i(T, t),\; T, t>0, i=0,1$. 
At each instant $\tau_n\in\mathfrak{T}$
the particle takes the velocity mode $c_{\ep(\tau_{n})}(T_{n}, \cdot)$, where
$T_{n}$ is the (random) time spent by the particle at the previous state
before the last switching.
 We define 
a generalised telegraph process  $\mathcal T=\mathcal T(t),\; t\geq0$ 
driven by the velocity modes $c_0,\; c_1$ as follows,
\begin{equation}\label{def:tp}
\mathcal T(t)=\mathcal T(t; c_0, c_1)
=c_{\ep(\tau_{n})}(T_{n}, t-\tau_{n}),\quad\text{if~~}\tau_{n}\leq t<\tau_{n+1},\; n\geq0.
\end{equation}

The integral 
$\int_0^t\mathcal T(s)\rmd s$  represents the current particle's position,
$\int_0^t\mathcal T(s)\rmd s$ is named the \emph{integrated telegraph process}.

Denote by $N=N(t):=\max\{n\geq0:~\tau_n\leq t\},\; t\geq0$ 
 a counting Poisson process.
Integrated telegraph processes can be described in terms of the compound Poisson process as follows.

Let $\tau_{n-1}\leq s<t<\tau_n,\; n\geq1$. 
Under the given value $i=\ep(\tau_{n-1})=\ep(s)$ denote
 the distance passed by the particle in time interval $(s,\; t)$
without any reversal by $l_i(T; s, t)$,
\begin{equation}
\label{def:l}
l_i(T; s, t)=\int_s^tc_{i}\left(T, u-\tau_{n-1}\right)\rmd u.
\end{equation}
Simplifying notations we will write $l_i(T; t)$ instead of $l_i(T; 0, t)$.

If $N (t)=0$, i. e. $0=\tau_0\leq t<\tau_1$, and $\ep(0)=i$, then the particle's position is 
\begin{equation}
\label{def:itp0}
\int\limits_0^t\mathcal T(u)\rmd u=l_i(T_0;  t).
\end{equation}
If  $N(t)>0$, then 
\begin{equation}
\label{def:itp}
\int\limits_0^t\mathcal T(u)\rmd u=\sum_{n=1}^{N(t)}l_{\ep(\tau_n)}(T_{n-1}; \tau_{n-1}, \tau_{n})
+l_{\ep(\tau_{N(t)})}(T_{N(t)}; \tau_{N(t)}, t).
\end{equation}
Equalities \eqref{def:l}-\eqref{def:itp} define the integrated telegraph process.

Similarly,  the jump component can be constructed.
Let  $h_0=h_0(T)$ and $h_1=h_1(T),$ $T\geq0,$ 
be  a pair of deterministic piecewise continuous 
(or, at least, boundary measurable) functions. 
Consider piecewise constant telegraph processes based on $h_i(T)$ 
instead of $c_i=c_i(T, \cdot),\; i=0,1$, see  \eqref{def:tp}:
\[
\mathcal T(t; h_0, h_1)= h_{\ep(\tau_{n})}(T_{n}),\quad \text{if~~}
\tau_{n}<t\leq\tau_{n+1},\; n\geq0.
\]
We define an integrated jump process as the compound Poisson process,
\begin{equation}
\label{def:ijp}
\int\limits_0^t\mathcal T(u; h_0, h_1)\rmd N(u)
=\sum_{n=1}^{N(t)}h_{\ep(\tau_{n})}(T_n).
\end{equation}
The amplitude of the subsequent jump depends on the time 
spent by the particle in the current state.

Generalised integrated jump-telegraph process 
is sum of the integrated telegraph process
defined by \eqref{def:itp0}-\eqref{def:itp} and the jump component defined by \eqref{def:ijp}:
\begin{equation}\label{def:ijtp}
X(t)=\int\limits_0^t\mathcal T(u; c_0, c_1)\rmd u
+\int\limits_0^t\mathcal T(u; h_0, h_1)\rmd N(u),\quad t\geq0.
\end{equation}
We consider also the processes $X_i,\; i=0,1$, defined by \eqref{def:ijtp}
under the fixed initial state of $\ep(0)=i\in\{0, 1\}$.
So, $X_i(t)$ 
gives the position at time $t,\; t\geq0$ 
of the particle, which starts at the origin with velocity mode $c_i,\; i=\ep(0),$
and continues moving with the alternated at random times $\tau_n$
velocity regimes.
 Each velocity reversal is accompanied by jumps of random amplitude.

Conditioning on the first switching, we have the following equalities in distribution
(under the probability $\PP_0$ and $\PP_1$ respectively):
\begin{equation}
\label{eq:X0}
X_0(t)\mid_{\PP_0}\stackrel{D}{=}l_0(T_{0}; t)\1_{\{\tau_1>t\}}+\left[l_0(T_{0}; \tau_1)
+h_0(\tau_1)+\tilde X_1(t-\tau_1)\right]
\1_{\{\tau_1<t\}},
\end{equation}
where $T_0$ and $\tau_1$ have the distribution functions $F_1$ and $F_0$ respectively;
\begin{equation}
\label{eq:X1}
X_1(t)\mid_{\PP_1}\stackrel{D}{=}l_1(T_{0}; t)\1_{\{\tau_1>t\}}+\left[l_1(T_{0}; \tau_1)
+h_1(\tau_1)+\tilde X_0(t-\tau_1)\right]
\1_{\{\tau_1<t\}},
\end{equation}
where $T_0$ and $\tau_1$ are distributed in the opposite order, with distribution functions
 $F_0$ and $F_1$ respectively.
Here $\tilde X_i(t)$ is the integrated jump-telegraph process
starting with the velocity regime $c_i(T_1; \cdot),\; i=0, 1.$

The distributions of $X_0(t), X_1(t)$ and $X(t), \;t>0,$ are separated into 
the singular and the absolutely continuous parts.
All distributions will be described in terms of the conditional probabilities $\PP_i(\cdot~|~N(s)=0)$
under  the condition $\{N(s)=0\}=\{\tau_1>s\}$, see \eqref{barFbarf}.
Here $s, \; s\in[0, \tau_1)$ is the time when the observations begin.

For any Borelian set $B\subset(-\infty, \infty)$ consider 
\[\PP_i(B,t|s):=\PP_i(X(t)\in B~|~N(s)=0),\quad  i=0,1.\]

The singular part of the distribution $\PP_i(\cdot,t|s)$
corresponds to 
the first terms in the RHS of \eqref{eq:X0}-\eqref{eq:X1}, the movement without any  reversal, $N(t)=0.$
To describe the singular part, consider
 the linear functionals (generalised functions),
\[\begin{aligned}
\phi\to
\PP_i(\tau_1>t|N(s)=0)\EE_i\{\phi(l_i(\tau_1; t))\}
=&\bar F_i(t|s)\int\limits_0^\infty f_{1-i}(\tau)\phi(l_i(\tau; t))\rmd \tau,\\ &i=0, 1,
\end{aligned}\] 
on the space of (continuous) test-functions $\phi.$ 
The generalised function 
\begin{equation}
\label{p0}
p_i^0(x, t|s)=\bar F_i(t|s)\int_0^\infty f_{1-i}(\tau)\delta_{l_i(\tau; t)}(x)\rmd \tau,\quad i=0, 1
\end{equation}
can be viewed as the (conditional) distribution  ``density''. 
Here $\delta_a(x)$ is the Dirac measure (of unit mass) at point $a$.

Let 
\[
p_i(x, t|s)=\PP_i\left(X(t)\in\rmd x~|~N(s)=0\right)/\rmd x
\]
be the distribution densities  of $X_i(t),\;  i=0, 1$. 

By conditioning on the first velocity reversal we obtain the following system 
of integral equations
\begin{equation}\label{eq:p01}\begin{aligned}
p_0(x, t|s)=&p_0^0(x, t|s)\\
+&\int_0^\infty f_1(\tau)\rmd\tau\int_s^t  p_1(x-l_{0}(\tau; u)-h_{0}(u), t-u)f_0(u|s)\rmd u,\\
p_1(x, t|s)=&p_1^0(x, t|s)\\
+&\int_0^\infty f_0(\tau)\rmd\tau\int_s^t p_0(x-l_{1}(\tau; u)-h_{1}(u), t-u)f_1(u|s)\rmd u,\\
&t>s\geq0,\; x\in(-\infty, \infty),
\end{aligned}\end{equation}
where $p_i^0,\; i=0, 1,$ are defined by \eqref{p0}.

Then, systems similar to \eqref{eq:p01}  can be derived for the expectations.

Let $\mu_i(t|s):=\EE _i\{X(t)~|~N(s)=0\},\; t>s\geq0$ be the conditional expectation
with respect to $\PP_i(\cdot|N(s)=0),$
and $\mu_i(t):=\EE_i\{X(t)\}=\lim\limits_{s\downarrow0}\mu_i(t|s),$  $i=0, 1$.

One can  easily obtain, for $t>s$  
\[
\mu_i(t|s)=\bar F_i(t|s)\bar l_i(t)+\int_s^t\left(\bar l_i(u)+h_i(u)+\mu_{1-i}(t-u)\right)f_i(u|s)\rmd u,
\]
where
 $\bar l_i(\cdot)=\EE\{l_i(T;\cdot)\}=\int_0^\infty f_{1-i}(\tau)l_i(\tau;\cdot)\rmd\tau,\;  i=0, 1$.
 
Therefore, the conditional expectations $\mu_0(t|s)$ and $\mu_1(t|s)$
read as,
\begin{equation}
\label{eq:mu1}\begin{aligned}
\mu_0(t|s)=&a_0(t|s)+\int_s^t\mu_1(t-u)f_0(u|s)\rmd u,\\
\mu_1(t|s)=&a_1(t|s)+\int_s^t\mu_0(t-u)f_1(u|s)\rmd u.
\end{aligned}\end{equation}
Hence, the expectations $\mu_i=\mu_i(t)=\EE_i\{X(t)\},\; i=0, 1$ satisfy the following   Volterra-type system
\begin{equation}
\label{eq:mu1V}\begin{aligned}
\mu_0(t)=&a_0(t)+\int_0^t\mu_1(t-u)f_0(u)\rmd u,\\
\mu_1(t)=&a_1(t)+\int_0^t\mu_0(t-u)f_1(u)\rmd u.
\end{aligned}\end{equation}
Here
\[
a_i(t|s):=\bar F_i(t|s)\bar l_i(t)+\int_s^t(\bar l_i(u)+h_i(u))f_i(u|s)\rmd u,\quad t>s,
\]
and $a_i(t)=a_i(t|0),\; i=0, 1.$
Integrating by parts in the latter integral, we have 
\[
\int_s^t\bar l_i(u)f_i(u|s)\rmd u=-\bar F_i(t|s)\bar l_i(t)+\bar l_i(s)+\int_s^t\bar c_i(u)\bar F_i(u|s)\rmd u, 
\]
which leads to the following expression for $a_i(t|s)$:
\[
a_i(t|s)=\bar l_i(s)+\int_s^t\left(
\bar F_i(u|s)\bar c_i(u)+f_i(u|s)h_i(u)
\right)\rmd u.
\]
Since $\bar l_i(0)=0$ and $\bar F_i(t|0)=\bar F_i(t),\; f_i(t|0)=f_i(t)$, we get
\begin{equation}
\label{eq:a0}
a_i(t)=\int_0^t\left(
\bar F_i(u)\bar c_i(u)+f_i(u)h_i(u)
\right)\rmd u.
\end{equation}

Moreover, the equalities in   \eqref{barFbarf} and \eqref{eq:a0} lead to
\begin{equation}
\label{eq:a}\begin{aligned}
a_i(t|s)=&\bar l_i(s)+\bar F_i(s)^{-1}\int_s^t\left(
\bar F_i(u)\bar c_i(u)+f_i(u)h_i(u)
\right)\rmd u\\
=&\bar l_i(s)+\bar F_i(s)^{-1}\left(a_i(t)-a_i(s)\right),\quad i=0, 1.
\end{aligned}\end{equation}

Here we denote 
$\bar c_i(s)=\EE\{c_i(\cdot; s)\}=\int_0^\infty f_{1-i}(\tau)c_i(\tau; s)\rmd \tau$.

\begin{rem}\label{mu--0}
The Volterra system \eqref{eq:mu1V} 
has a unique solution, see e. g. \cite{linz}.
Moreover,   
$\mu_0(t)\equiv0$, $\mu_1(t)\equiv0$ 
if and only if  $a_0(t)\equiv0, a_1(t)\equiv0$, or equivalently,
\begin{equation}\label{eq:DoobMeyer}
\begin{aligned}
\bar F_0(t)\bar c_0(t)+h_0(t)f_0(t)=0\\
\bar F_1(t)\bar c_1(t)+h_1(t)f_1(t)=0
\end{aligned}\;,\qquad t\geq0,
\end{equation}
see \eqref{eq:a0}. 

Hence, $\mu_0(t)\equiv0$, $\mu_1(t)\equiv0$ if and only if 
the hazard rate functions of the spending time $T,\; T>0,$          see e.g. \cite{Cap},
 \begin{equation}
\label{def:hazard}
\alpha_i(t):=\frac{f_i(t)}{\bar F_i(t)}
\end{equation}
 are expressed by
$\alpha_i(t)=-\bar c_i(t)/h_i(t), \; t\geq0.$

Due to equations \eqref{eq:mu1} and \eqref{eq:a0}-\eqref{eq:a},
condition \eqref{eq:DoobMeyer} guarantees that 
\begin{equation}
\label{eq:muts}
\mu_0(t|s)=\bar l_0(s),\qquad \mu_1(t|s)=\bar l_1(s).
\end{equation}
\end{rem}

In some particular cases the solution of
\eqref{eq:mu1}  and \eqref{eq:mu1V} 
can be written explicitly.
Consider the following example. Let the alternated 
distributions of interarrival times are exponential:
\begin{equation}\label{def:exp}
f_i(t)=\lambda_i\exp(-\lambda_it),\quad t\geq0,\; i=0, 1.
\end{equation}
Hence $f_i(t|s)=\lambda_i\exp(-\lambda_i(t-s)),\; t>s\geq0.$
In this case the solution of system \eqref{eq:mu1V}  reads 
\begin{equation}
\label{eq:mu}
\boldsymbol{\mu}(t)
=\boldsymbol{a}(t)+\int_0^t\left(I+\varphi_\lambda(t-u)\Lambda\right)L\boldsymbol{a}(u)\rmd u,
\end{equation}
where 
\begin{equation}
\label{def:varphi}
\varphi_\lambda(t)=\frac{1-\rme^{-2\lambda t}}{2\lambda},\quad 2\lambda:=\lambda_0+\lambda_1.
\end{equation}
Here we use the matrix notations
$\boldsymbol{\mu}=(\mu_0, \mu_1)',$
$\boldsymbol{a}=(a_{0}, a_{1})'$, see \eqref{eq:a0},
 \[
 L=\begin{pmatrix}
0  & \lambda_0   \\
\lambda_1  &  0  
\end{pmatrix}\qquad \text{and}\qquad
\Lambda=\begin{pmatrix}
  -\lambda_0    &   \lambda_0  \\ 
  \lambda_1    &  -\lambda_1
\end{pmatrix}
.\]

To check it, notice that system \eqref{eq:mu1V} is equivalent to ODE with zero initial condition:
\[
\frac{\rmd \boldsymbol{\mu}(t)}{\rmd t}=\Lambda\boldsymbol{\mu}(t)+\boldsymbol{\psi}(t),\quad t>0,
\qquad\boldsymbol{\mu}(t)\mid_{t\downarrow 0}=\boldsymbol{0},
\]
where $\boldsymbol{\psi}=\dfrac{\rmd\boldsymbol{a} }{\rmd t}+(L-\Lambda)\boldsymbol{a}.$
We obtain this equation  differentiating in \eqref{eq:mu1V}
 with subsequent integration by parts.
Clearly, the equation is solved by 
\begin{equation}
\label{sol:mu}
\boldsymbol{\mu}(t)=\int_0^t\rme^{(t-u)\Lambda}\boldsymbol{\psi}(u)\rmd u.
\end{equation}
Integrating by parts in \eqref{sol:mu} we obtain 
\[
\boldsymbol{\mu}(t)=\boldsymbol{a}(t)+\int_0^t\rme^{(t-u)\Lambda}L\boldsymbol{a}(u)\rmd u.
\]
Since $\Lambda^2=-2\lambda\Lambda$,  the exponential of $t\Lambda$ is
\begin{equation*}\label{def:expL}
\exp\{t\Lambda\}={\rm I}+\varphi_\lambda(t)\Lambda=\frac{1}{2\lambda}
\begin{pmatrix}
 \lambda_1+\lambda_0\rme^{-2\lambda t}     &\quad   \lambda_0(1-\rme^{-2\lambda t})   \\ \\
 \lambda_1(1-\rme^{-2\lambda t})     & \quad  \lambda_0+\lambda_1\rme^{-2\lambda t}  
\end{pmatrix},
\end{equation*}
and then, we have \eqref{eq:mu}.

The explicit formulae for conditional expectations $\mu_i(t|s),\; i=0,1$ follow  directly
from \eqref{eq:mu1} and \eqref{eq:mu}.

Equations for variances 
$\sigma_i(t):=\mathrm{Var}\{X_i(t)\}=\EE\{\left(X_i(t)-\mu_i(t)\right)^2\},\; t>0,$ 
have the form, similar to \eqref{eq:mu1V}:
\begin{equation}
\label{eq:var}\begin{aligned}
\sigma_0(t)=&b_0(t)+\int_0^t\sigma_1(t-u)f_0(u)\rmd u,\\
\sigma_1(t)=&b_1(t)+\int_0^t\sigma_0(t-u)f_1(u)\rmd u,
\end{aligned}\end{equation}
where
\[\begin{aligned}
b_i(t):=&\bar F_i(t)\left(\bar l_i(t)-\mu_i(t)\right)^2\\
+&\int_0^t\left(\bar l_i(u)+h_i(u)+\mu_{1-i}(t-u)-\mu_i(t)\right)^2f_i(u)\rmd u,\quad
 i=0, 1.
\end{aligned}
\]
In the special case of exponential distributions \eqref{def:exp}
the solution of \eqref{eq:var} reads similar to \eqref{eq:mu}.

 More specifically, the solution of \eqref{eq:var} is given by
\begin{equation}\label{bfsigma}
\boldsymbol{\sigma}=\boldsymbol{b}(t)+\int_0^t(I+\varphi_\lambda(t-u)\Lambda)L\boldsymbol{b}(u)\rmd u,
\end{equation}
where $\boldsymbol{\sigma}=(\sigma_0, \sigma_1)',\; \boldsymbol{b}=(b_0, b_1)'$. We use also 
the notations of \eqref{eq:mu} and \eqref{def:varphi}.

\begin{rem}
Let $0= \tau_0< \tau_1< \tau_2<\ldots$ be a Poisson univariate point process
 with deterministic constant intensity $2\lambda,\; \lambda>0$.
 Let $Y_n, n\geq1$ be a sequence of the i. i. d. random variables with two values, $Y_n\in\{0, 1\}$.
 We may then consider two counting processes
 \[
 N^{(0)}(t):=\sum_{n\geq1}\1_{\{\tau_n\leq t\}}\1_{\{Y_n=0\}},\quad
 N^{(1)}(t):=\sum_{n\geq1}\1_{\{\tau_n\leq t\}}\1_{\{Y_n=1\}}.
 \]
 It is easy to see that  $ N^{(i)}(t),\; t\geq0$ is a univariate point process with intensity $\lambda_i,\; i=0, 1.$
Here $\lambda_0=2\lambda(1-p)$ and $\lambda_1=2\lambda p,$ where $p=\PP(Y_n=1).$
In the special case \eqref{def:exp} the Markov flow $\mathfrak T$ is the bivariate point process $(N^{(0)}(t), N^{(1)}(t))$.
See \cite{Bremaud}.
\end{rem}

\section{Martingales}
\label{sec:mart}
\setcounter{equation}{0}

Let $X=X(t)$  be integrated jump-telegraph process
defined  by \eqref{def:ijtp}
on the filtered probability space $(\Omega, \mathcal{F}, \{\mathcal{F}_t\}_{t\geq0}, \PP)$.

\begin{theo}\label{theo1}
Process $X$  
is $\mathcal{F}_t$-martingale if and only if \eqref{eq:DoobMeyer} holds.
\end{theo}

\begin{proof}
We need to show that \eqref{eq:DoobMeyer} is necessary and sufficient for 
\begin{equation}
\label{eq:mart}
\EE\left\{X(t)~|~\mathcal{F}_{s}\right\}=X(s),\quad   0<s<t.
\end{equation}

First, notice that
for any stopping times $t$ and $s$ such that, $s<t$, 
we have 
\[
\EE_i\{X(t)-X(s)~|~\F_{s}\}
=\EE_i\left\{\int_{s}^{t}\mathcal T(u; c_0, c_1)\rmd u+
\sum_{k=N(s)+1}^{N(t)}
h_{\ep(\tau_k)}(T_k)~|~\F_{s}\right\}
\]
\[
=\EE_i\left\{
\int_0^{t-s}\mathcal T(s+u)\rmd u+
\sum_{k=1}^{N(t)-N(s)}h_{\ep(\tau_{k+N(s)})}(T_{k+N(s)})~|~\F_{s}\right\}
\]
 on the set $0\leq s\leq t$.
 
Let $s,\; s\geq 0$ be a switching time, $s=\tau_n$
(in this case the proof is similar to \cite{R13}).
According to the Markov property by  definition of the processes 
$\ep=\ep(t),\; N=N(t)$ and $\tau_k$ 
we have the following identities in (conditional) distribution
\[\begin{aligned}
\ep(\tau_n+u)|_{\{\ep(\tau_n)=i\}}\stackrel{D}{=}
&\tilde \ep(u)|_{\{\tilde\ep(0)=i\}},\quad   &u\geq0,\\
N(t)|_{\{\ep(\tau_n)=i\}}\stackrel{D}{=}
& n+\tilde N(t-\tau_n)|_{\{\tilde\ep(0)=i\}},\quad    &t\geq \tau_n\geq0,\\
\tau_{k+n}|_{\{\ep(\tau_n)=i\}}\stackrel{D}{=}&\tilde \tau_k|_{\{\tilde\ep(0)=i\}},\;
T_{k+n}|_{\{\ep(\tau_n)=i\}}\stackrel{D}{=}\tilde T_k|_{\{\tilde\ep(0)=i\}},& \quad k\geq0,
\end{aligned}\]
where $\tilde \ep(s),\; \tilde N(s),\; \tilde\tau_k$ and $\tilde T_k$ 
are copies of 
$\ep(s),\; N(s),\; \tau_k$ and $T_k$ respectively,
independent of $\mathcal F_{\tau_n}$. Hence, we obtain
\begin{equation*}\label{eq:EE}
\EE\{X(t)-X(\tau_n)~|~\F_{\tau_n}\}
= \EE_i\{\tilde X(t-\tau_n)\},
\end{equation*}
if $\ep(\tau_n)=i$.
Here $\tilde X$ denotes the integrated jump-telegraph process, 
which is based on  $\tilde \ep,\; \tilde N,\; \tilde\tau_k$ and $\tilde T_k$,
 starting from the state $\tilde\ep(0)=i$.
 The latter expectation is equal to zero,
$\EE_i\{\tilde X(t-\tau_n)\}\equiv0$, if and only  if
\eqref{eq:DoobMeyer} holds. Thus, equality \eqref{eq:mart} is proved, if $s$ is the  
switching time, $s=\tau_n$.

In general, for any $s,\; s<t$ the martingale property \eqref{eq:mart} is proved by using \eqref{eq:muts},
Remark \ref{mu--0}.
\end{proof}

\begin{rem}\label{rem:signs}
Notice that if $X$ is the martingale, so
identities
\eqref{eq:DoobMeyer} hold, then the direction of each  jump should be opposite 
to the respective \textup{(}mean\textup{)} velocity value.
\end{rem}

\begin{rem}
In the special case of process $X$
 with exponential distributions of interarrival times \eqref{def:exp}
the set of equalities \eqref{eq:DoobMeyer} is equivalent to 
\[
\bar c_0(t)+\lambda_0h_0(t)=0,\qquad
\bar c_1(t)+\lambda_1h_1(t)=0,
\]
see also Theorem 1 in \cite{R2007}.
\end{rem}

\begin{theo}\label{cor}
Let the jump-telegraph process $X$ 
be defined by \eqref{def:ijtp}\textup{,}
 and $h_i\neq0,$  $ i=0, 1$.
If $X$ is the martingale\textup{,}  then 
\begin{align}\label{cond:DM}
\frac{\bar c_i(t)}{h_i(t)}<&0 \qquad\forall t>0,\\
\label{cond:dens}
\int_0^\infty\frac{\bar c_i(s)}{h_i(s)}\rmd s=&-\infty, \quad i=0, 1.
\end{align}

Moreover,
$X$ is the martingale, if and only if the 
 the hazard rate functions $\alpha_i(t)$ \textup{(}\,see the definition in \eqref{def:hazard}\textup{)}
of interarrival times are expressed by 
\begin{equation}\label{eq:hazard}
\alpha_i(t)=-\bar c_i(t)/h_i(t), \; t\geq0.
\end{equation}
Therefore,
the distribution densities of interarrival times  
 satisfy the following  set of integral equations\textup{:}
\begin{equation}\label{eq:fi}
f_i(t)=\alpha_i(t)\exp\left\{-\int_0^t\alpha_i(s)\rmd s
\right\}
\equiv-\frac{\bar c_i(t)}{h_i(t)}\exp\left\{\int_0^t\frac{\bar c_i(s)}{h_i(s)}\rmd s\right\},\qquad  t>0,\; i=0, 1.
\end{equation}
\end{theo}

\begin{proof}
Equations \eqref{eq:hazard} are derived, see Theorem \ref{theo1}. By these equations
\begin{equation}\label{barch}
-\frac{\bar c_i(t)}{h_i(t)}=\alpha_i(t)=\frac{f_i(t)}{\bar F_i(t)}\equiv -(\ln \bar F_i(t))',\quad
t\geq0,\; i=0, 1.
\end{equation}
Thus the survival probability is
\[
\bar F_i(t)=\exp\left\{\int_0^t\frac{\bar c_i(s)}{h_i(s)}\rmd s\right\},\qquad t\geq0,\; i=0, 1.
\]
So, the density is given by  \eqref{eq:fi}. 

Inequality \eqref{cond:DM}  follows   from \eqref{eq:DoobMeyer}.
Notice that by definition $\lim\limits_{t\to+\infty}\bar F_i(t)=0,$    
 hence  \eqref{cond:dens} is valid.\end{proof}

Consider  the following example.
Assume that functions $\bar c_i(t)$ and $h_i(t)$ are proportional\textup{:}
\begin{equation}\label{eq:barchlambda}
\frac{\bar c_i(t)}{h_i(t)}\equiv-\lambda_i,\qquad \lambda_i>0,\; i=0, 1.
\end{equation}
Therefore, by \eqref{eq:fi} 
 the respective  integrated jump-telegraph process
is the martingale if the
 distributions of interarrival times are \emph{exponential} 
 with densities $f_i(t)=\lambda_i\exp(-\lambda_it),$  $t>0,\; i=0, 1.$ 
 
Identities \eqref{eq:barchlambda} can be written in detail as follows.
Let $X$  be the jump-telegraph process with regimes of velocities $c_0, c_1$
and the regimes of jumps $h_0, h_1$, which
are connected by means of the relations
\[
\lambda_{1}\int_0^\infty\rme^{-\lambda_{1}\tau}c_0(\tau, t)\rmd \tau
=-\lambda_0h_0(t),\qquad
\lambda_{0}\int_0^\infty\rme^{-\lambda_{0}\tau}c_1(\tau, t)\rmd \tau
=-\lambda_1h_1(t).
\]
Here  $\lambda_0$ and $\lambda_1$ are some positive constants. 
Hence the jump-telegraph process
$X$ is the martingale with exponentially distributed interarrival times.
Parameters of these alternated exponential distributions are $\lambda_0$ and $\lambda_1$.

Equations \eqref{eq:barchlambda}
permit to interpret 
the switching intensities $\lambda_0$ and $\lambda_1$  
by using the (observable)  proportion between velocity  and jump values. 
On the other hand, if the average velocity regimes are given,
$\bar c_0$ and $\bar c_1$,   and $X_0$ and $X_1$ are martingales, then  
we can observe the details of comportment of process $X=X(t)$. 
For example, the martingale possesses small jumps  with high frequency, while the big jumps are rare.
The direction of jump should be opposite to the velocity sign, see  also Remark \ref{rem:signs}.

Other useful examples are presented in \cite{R13}, see Examples 1-4, pp.2289-2290.

\begin{prop}\label{prop}
Let $\mathfrak{T}=\{\tau_n\}$ be 
the Markov flow of switching times,
and  $X$ 
be a jump-telegraph process defined by  \eqref{def:ijtp}.
Suppose  that the increments
 $T_n=\tau_n-\tau_{n-1},\; n\geq1$
are exponentially distributed with alternated 
parameters $\mu_0, \mu_1>0$. 

Assume  that  the velocity regimes $c_i=c_i(T, t)$ and the
 jump amplitudes $h_i=h_i(t)$ are  proportional
satisfying \eqref{eq:barchlambda} 
 with some positive coefficients $\lambda_0$ and $\lambda_1$.
 
 Therefore the martingale measure for $X$  exists and it is unique.
 Under the martingale measure the interarrival times are exponential with parameters 
 $\lambda_0$ and $\lambda_1$.
\end{prop}

\begin{proof} 
For the integrated jump-telegraph process defined by \eqref{def:ijtp} we define 
the Radon-Nikodym derivative of the form, see \cite{R2007}
\begin{equation}\label{def:RN}
\frac{\rmd \QQ}{\rmd \PP}
=\exp\left\{\int_0^t\mathcal T(u; c_0^*, c_1^*)\rmd u\right\}
\kappa^*(t).
\end{equation}
Here $X^*=X^*(t)$ is the jump-telegraph process  driven by the   Markov flow $\mathfrak{T}$
(with parameters $\mu_0>0$ and $\mu_1>0$). 
Process $X^*$ is defined by the  constant velocities 
$c_0^*=\mu_0-\lambda_0$ and $c_1^*=\mu_1-\lambda_1$ and
the constant jump parameters $h_0^*=-c_0^*/\mu_0, \; h_1^*=-c_1^*/\mu_1.$ 
The jump part
$\kappa^*(t)=\prod_{n=1}^{N(t)}(1+h_{\ep(\tau_{n-1})}^*)$ follows from the exponential formula, \cite{JDM},
and $\int_0^t\mathcal T(u; c_0^*, c_1^*)\rmd u$ is the integrated telegraph process. 

Since Theorem 2 and Theorem 3 in   \cite{R2007},  under the new measure $\QQ$
the underlying Markov flow takes the intensities $\lambda_i$ instead of $\mu_i,\; i=0, 1$, 
see also \cite{cheang} (Lemma1) and \cite{JDM} (Theorem 2.5).

Since, due to  \eqref{eq:barchlambda}, condition \eqref{eq:DoobMeyer}
is fulfilled. By Theorem \ref{cor} the process $X(t)$ becomes the $\QQ$-martingale. 
\end{proof}
  
\section{Market model and fundamental equation}\label{sec:3}
\setcounter{equation}{0}    

Let the price process $S=S(t)$  be  defined by stochastic exponential of  
generalised jump-telegraph process $X$. 
The velocity and jump regimes are established in accordance with  time 
spent by the process in the previous state (see the definition in  \eqref{def:ijtp}). 
 
 Precisely, let
  $\ep=\ep(t)\in\{0, 1\},\; t\in[0, U]$ be the Markov process describing  
the evolution of market states.
Let $\mathfrak{T}=\{\tau_n\}$ be 
the  flow of switching times.
Consider the integrated jump-telegraph process $X=X(t)$ based on $\ep$ and $\mathfrak T$,
which is defined by 
 \eqref{def:ijtp}
with  velocity regimes $c_0=c_0(T, t),\; c_1=c_1(T, t)$ and
jump amplitudes  $h_0(T),\; h_1(T)>-1, \;\forall T\geq0$. 

Consider a market model of one risky asset associated with price process $S=S(t)$, 
which is the stochastic exponential, 
\begin{equation}
\label{def:S}
S(t)=S(0)\mathcal{E}_t\{X\}
=S(0)\exp\left\{\int_0^t\mathcal T(s; c_0, c_1)\rmd s\right\}
\kappa(t),\quad t\in[0, U].
\end{equation}
Here 
$\kappa(t)=\prod_{n=1}^{N(t)}(1+h_{\ep(\tau_{n-1})}(T_n))$ 
is the jump component
of stochastic exponential $\mathcal{E}_t\{X\}$, 
see the exponential formula of Stieltjes-Lebesgue calculus
(\cite{Bremaud}, Theorem T4 in Appendix A4;
see also  \cite{JDM}, formula (17)).

Let $r_0=r_0(T, t)\geq0,\; r_1=r_1(T, t)\geq0,\; T, t\geq0$ be 
piecewise continuous  deterministic functions.
The bond price is assumed to be 
\begin{equation}
\label{def:B}
B(t)=\exp\left\{\int_0^t\mathcal{T} (u; r_0, r_1)\rmd u\right\}.
\end{equation}
Here $\mathcal T (\cdot; r_0, r_1)$ 
is  the telegraph process  driven by the same Markov process $\ep$
(see \eqref{def:tp}) and
$r_0,\; r_1$ are the interest rate functions. Thus,
the discounted price process is of the same structure as $S(t)$, \eqref{def:S},
\[
B(t)^{-1}S(t)
=S(0)\exp\left\{\int_0^t\mathcal T(u; c_0-r_0, c_1-r_1)\rmd u\right\}
\kappa (t).
\]
Hence, without loss of generality we assume   the interest rates to be 0.

Let $\QQ$ be the martingale measure for process $S$,    \eqref{def:S}.

Consider an option with the payoff function 
$\mathcal{H}=\mathcal{H}(x),\; \mathcal H(x)\geq0$ at the maturity time $U,\; U>0$.

Let $A_{i}(t|s, \rmd s):=\{\ep(t)=i,\; t-\tau_{N(t)}\in(s, s+\rmd s)\},  t\in(0, U),$ $i=0, 1.$
Here $s\in(0, t),$ and $\tau_{N(t)}$ is  the last switching time. 
Notice that  $A_{i}(t|s, \rmd s)\in\mathcal{F}_t.$  
Consider the functions
\begin{equation*}
\label{def:Phitsds}
\begin{aligned}
\Phi_i(x, t|s, \rmd s)=&\EE_\QQ\left\{
\mathcal{H}(x\rme^{\int_t^U\mathcal T(u; c_0, c_1)\rmd u}\kappa(U)/\kappa(t))
~|~A_{i}(t|s, \rmd s)
\right\}\QQ(A_i(t|s, \rmd s))\\
=&\int\limits_{A_i(t|s, \rmd s)}
\mathcal{H}(x\rme^{\int_t^U\mathcal T(u; c_0, c_1)\rmd u}\kappa(U)/\kappa(t))
\rmd \QQ,\\
&0\leq s<t\leq U,\; i=0, 1,
\end{aligned}\end{equation*}
see \eqref{barFbarf}.
Further, let
$\Phi_i(x, t|s)=\lim\limits_{\rmd s\downarrow0}\Phi_i(x, t|s, \rmd s).$ 

Notice that the strategy value at time $t\in(0, U)$ equals to
\[V(t|s)=\Phi_{\ep(t)}(S(t), t|s),\]
where $s=t-\tau_{N(t)}$ is the elapsed time since the last switching.

Conditioning on the first reversal after time $t$,
 we  see the explicit expressions for functions $\Phi_0(x, t|s)$ and $\Phi_1(x, t|s)$,
\begin{equation}
\label{def:Phits}
\begin{aligned}
\Phi_0(x, t|s)=&\bar F_0(U-t+s)\EE\left\{\mathcal{H}(x\rme^{l_0(\tau; t, U)})\right\}\\
+&\EE\left\{\int_t^Uf_0(u-t+s)\mathrm\Phi_1(x\rme^{l_0(\tau; t, u)}(1+h_0(u-t+s)), u)\rmd u\right\},\\
\Phi_1(x, t|s)=&\bar F_1(U-t+s)\EE\left\{\mathcal{H}(x\rme^{l_1(\tau; t, U)})\right\}\\
+&\EE\left\{\int_t^Uf_1(u-t+s)\mathrm\Phi_0(x\rme^{l_1(\tau; t, u)}(1+h_1(u-t+s)), u)\rmd u\right\}.
\end{aligned}\end{equation}
Here $\mathrm\Phi_0(\cdot, t)$ and $\mathrm\Phi_1(\cdot, t)$ are defined by
$\mathrm\Phi_i(x, t)=\lim\limits_{s\downarrow0}\Phi_i(x, t|s),\; i=0,1$.
Functions $\mathrm\Phi_i(\cdot, t),\; i=0, 1$ correspond to the market process initiated 
exactly at the switching time.

Finally, notice that   functions $\mathrm\Phi_0(\cdot, t)$ and $\mathrm\Phi_1(\cdot, t)$
solve the following Volterra system:
\begin{equation}
\label{def:Phit}
\begin{aligned}
\mathrm\Phi_0(x, t)=&\bar F_0(U-t)\EE\left\{\mathcal{H}(x\rme^{l_0(\tau; t, U)}\right\}\\
+&\EE\left\{\int_t^Uf_0(u-t)\mathrm\Phi_1(x\rme^{l_0(\tau; t,u)}(1+h_1(u-t)), u)\rmd u
\right\},\\
\mathrm\Phi_1(x, t)=&\bar F_1(U-t)\EE\left\{\mathcal{H}(x\rme^{l_1(\tau; t, U)}\right\}\\
+&\EE\left\{\int_t^Uf_1(u-t)\mathrm\Phi_0(x\rme^{l_1(\tau; t, u)}(1+h_0(u-t)), u)\rmd u
\right\}.
\end{aligned}\end{equation}
The set of integral equations  \eqref{def:Phits}-\eqref{def:Phit}
can be interpreted as the fundamental equation of the market model \eqref{def:S}-\eqref{def:B}.
In the case of 
deterministic and constant velocities and jumps these equations are equivalent to a hyperbolic 
PDE-system,  see equation (36) in \cite{R2007}.

\begin{rem}
Consider the model \eqref{def:S}-\eqref{def:B}
with constant $c_i, h_i$ and $r_i,\; i=0,1$
in the special case of exponentially distributed iterarrival times, \eqref{def:exp}.
The fundamental equations \eqref{def:Phits}-\eqref{def:Phit}  take the form of PDE-system:
\begin{equation}\label{eq:FEd}
\begin{aligned}
\frac{\partial \mathrm\Phi_i}{\partial t}(t, x) +c_i
x\frac{\partial \mathrm\Phi_i}{\partial x}(t,  x)
&=(r_i+\lambda_i)\mathrm\Phi_i(t, x)
-\lambda_i\mathrm\Phi_{1-i}(t, x(1+h_i)),\\
&0<t<T,\quad i=0,\; 1.
\end{aligned}
\end{equation}
Equation \eqref{eq:FEd}
is supplied with the terminal condition
\begin{equation}\label{bv}
\mathrm\Phi_i (x,\;  T) =\mathcal H(x).
\end{equation}
\end{rem}

\section{Memory effects. Numerical results}
\label{sec:num}
\setcounter{equation}{0}

\begin{figure}[p]
\begin{center}
\includegraphics[scale=0.25]{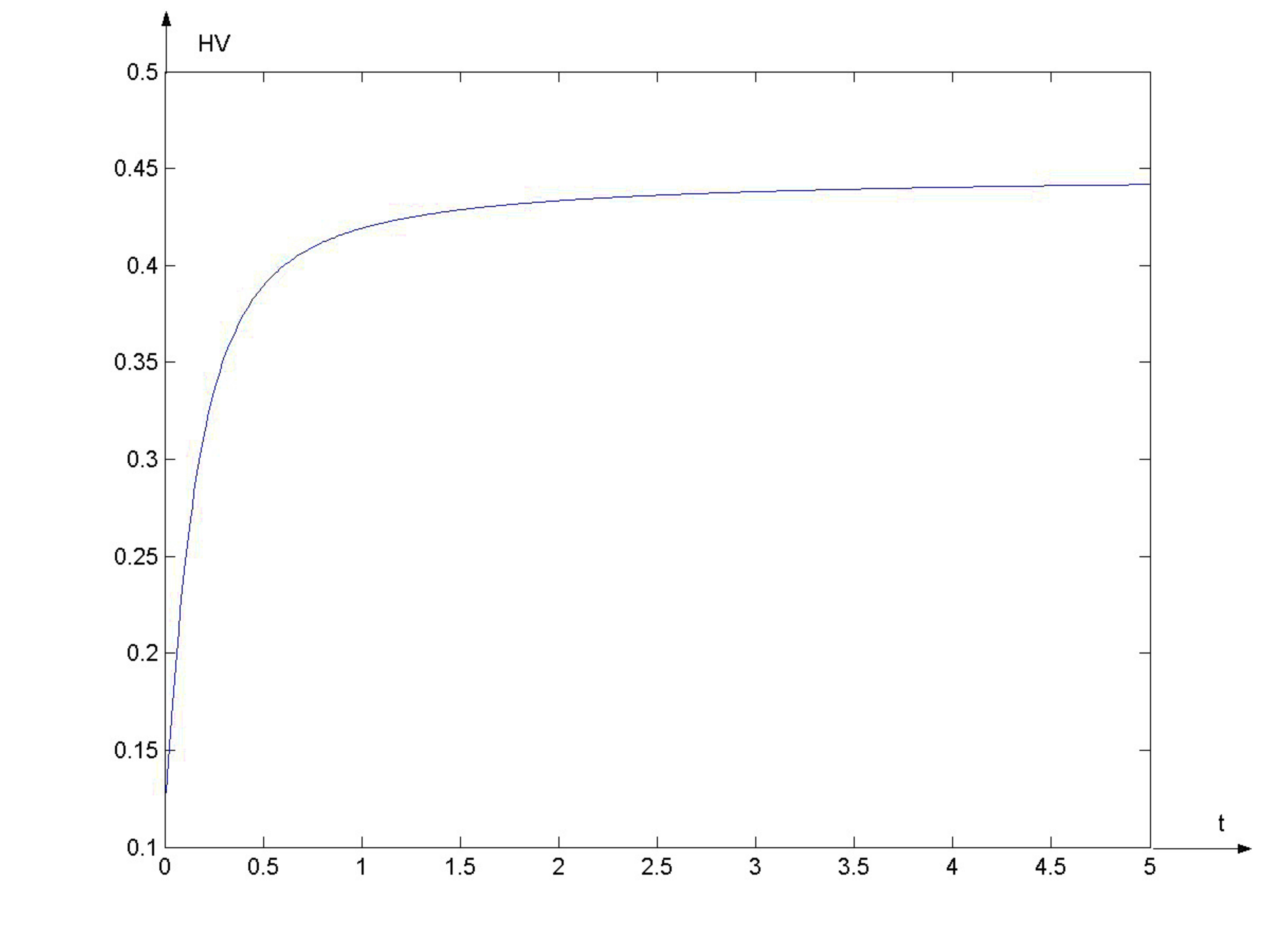}\\
\caption{Historical volatility under the  constant velocities and jump amplitudes (symmetric case):
 $c_0=1,\; h_0=-0.05;\;$ $c_1=-1,\; h_1=0.05;\;$ 
$\lambda_0=\lambda_1=5$ } \label{fig1}
\end{center}
\end{figure}

\begin{figure}[p]
\begin{center}
\includegraphics[scale=0.25]{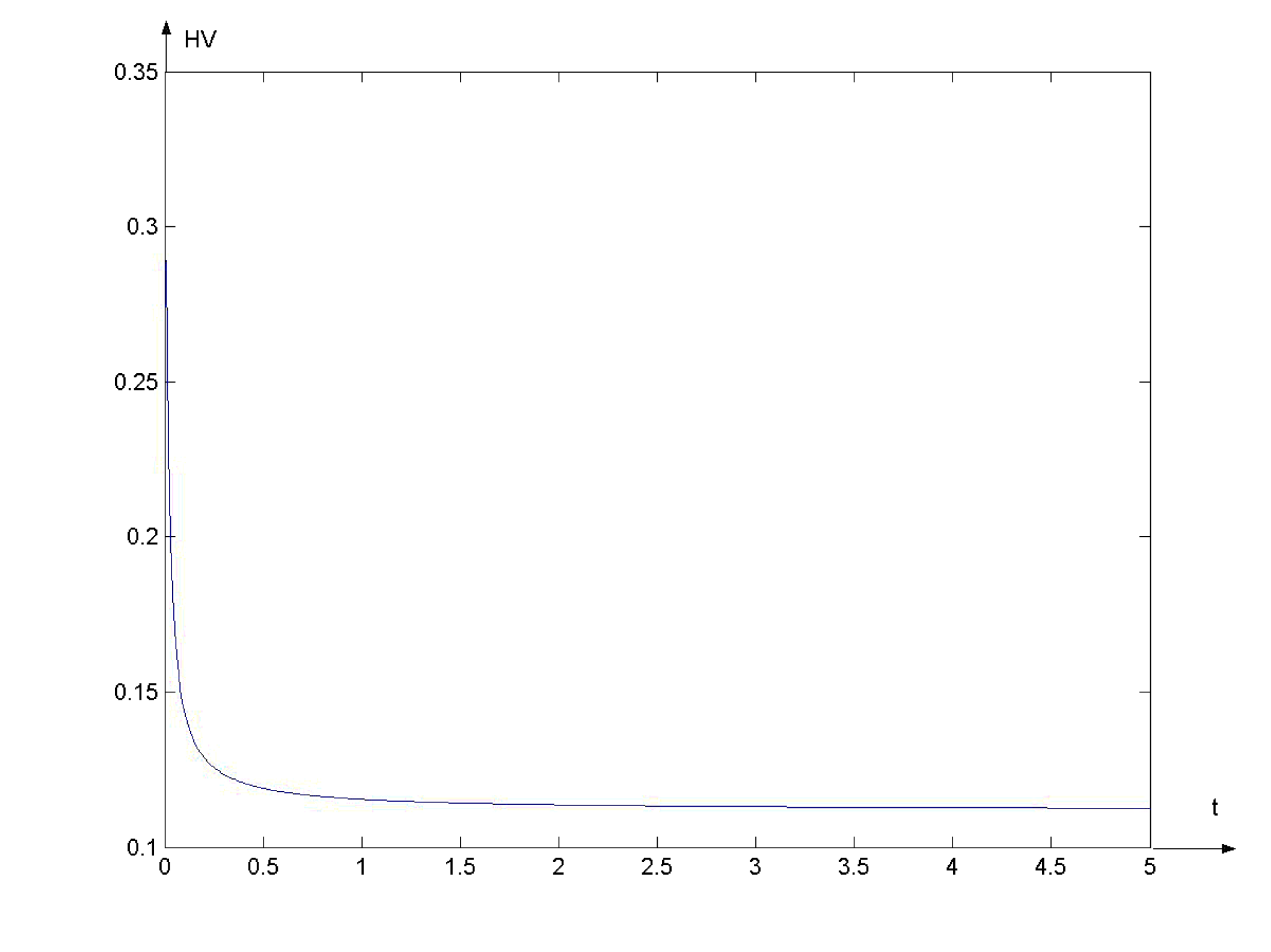}
  \caption{Historical volatility under the  constant velocities and jump amplitudes (symmetric case):
 $c_0=1,\; h_0=-0.05;\;$ $c_1=-1,\; h_1=0.05;\;$ 
   $\lambda_0=\lambda_1=80$ }
 \label{fig2}
 \end{center}
\end{figure}

\begin{figure}[p]
\begin{center}
\includegraphics[scale=0.25]{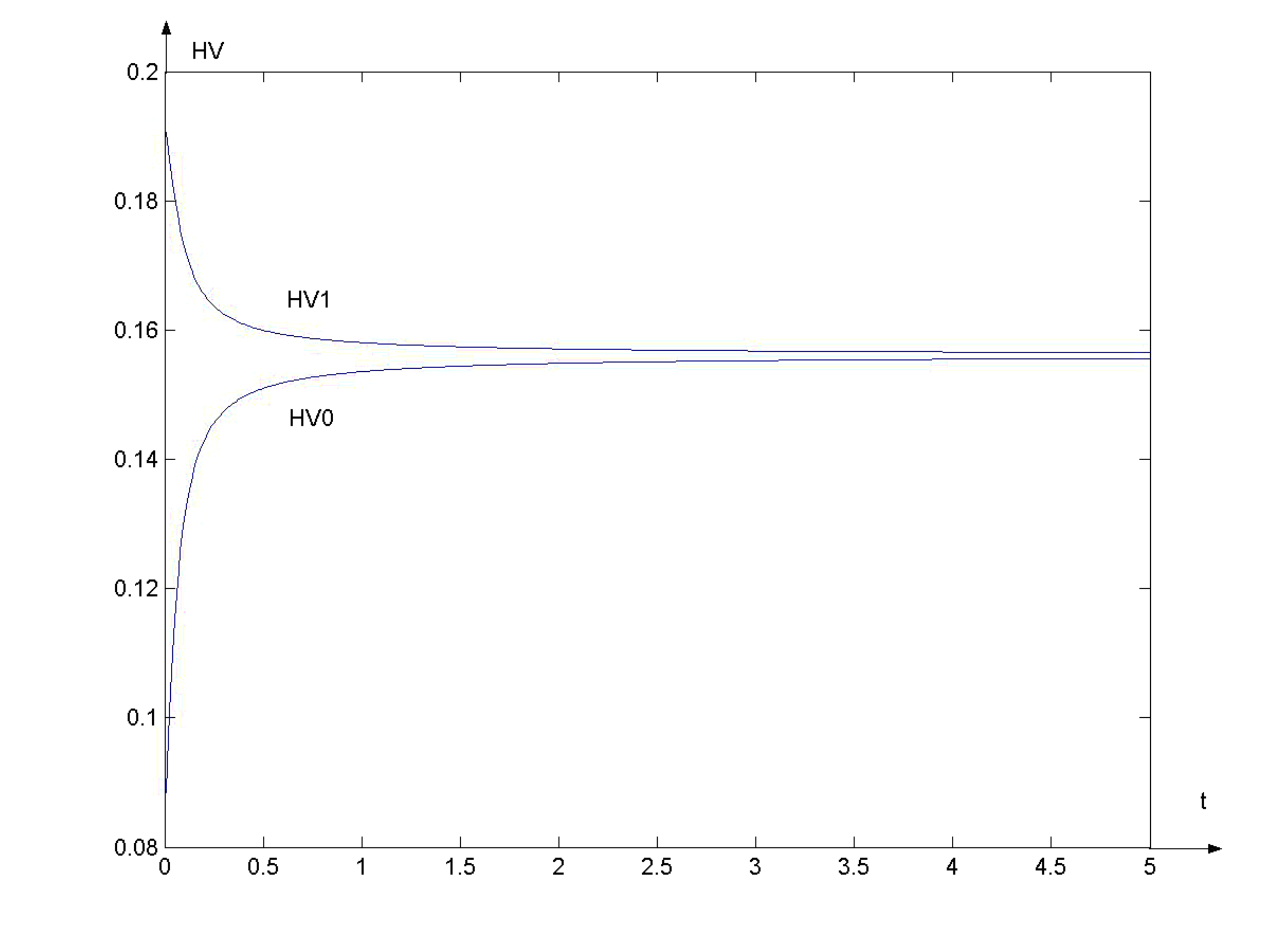}
 \caption{Historical volatility under the  constant velocities and jump amplitudes:
 $c_0=1.2,\; h_0=-0.05;\;$ $c_1=0.6,\; h_1=-0.02;\;$ 
$\lambda_0=\lambda_1=15$}
 \label{fig3}
 \end{center}
\end{figure}

\begin{figure}[p]
\begin{center}
\includegraphics[scale=0.25]{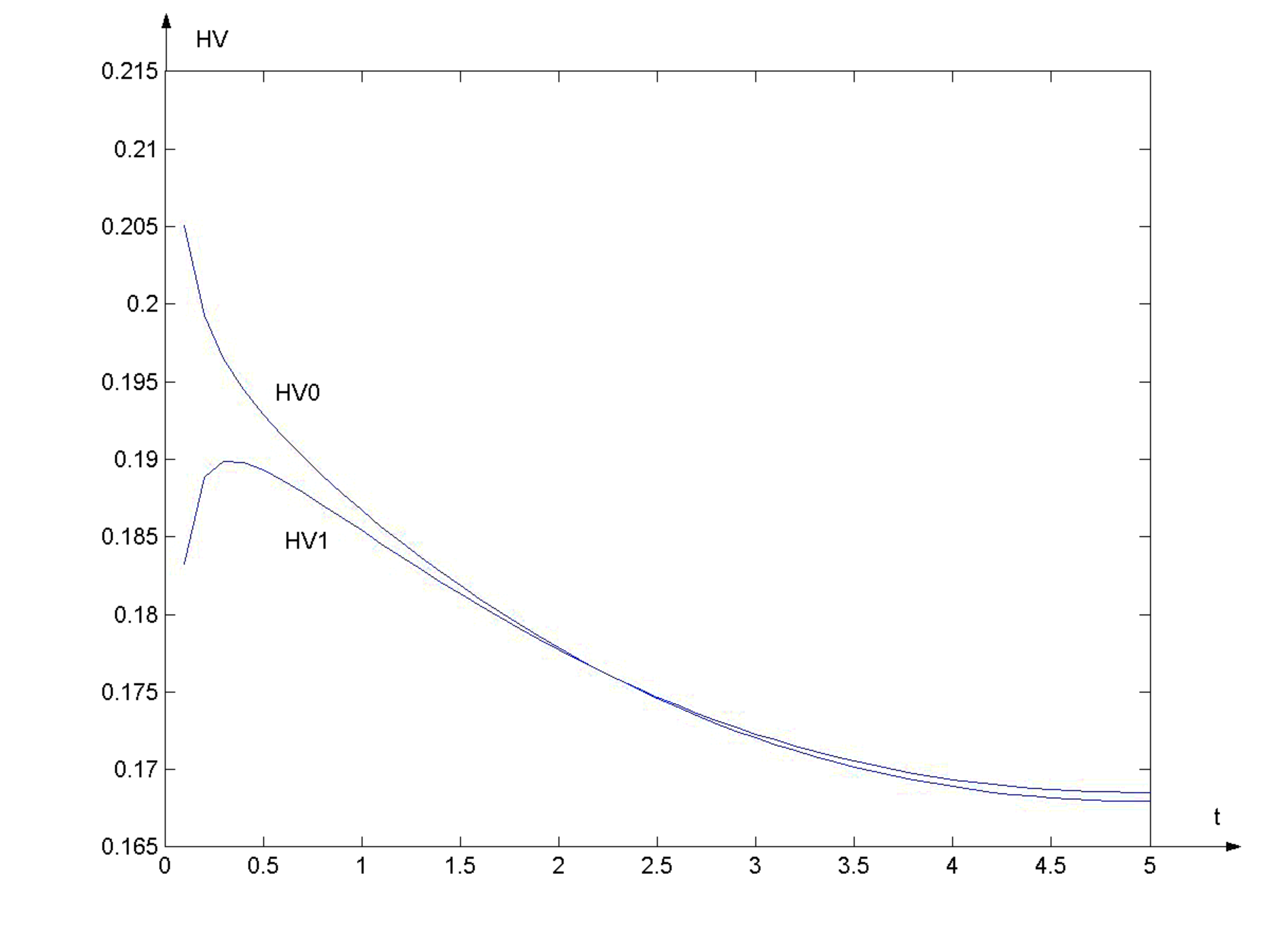}
\caption{Historical volatility under the  constant velocities and jump amplitudes:
 $c_0=1.2,\; h_0=-0.05;\;$ $c_1=0.6,\; h_1=-0.02;\;$ 
 $\lambda_0=24, \lambda_1=30$}
 \label{fig4}
 \end{center}
\end{figure}

\begin{figure}[p]
\begin{center}
  \includegraphics[scale=0.25]{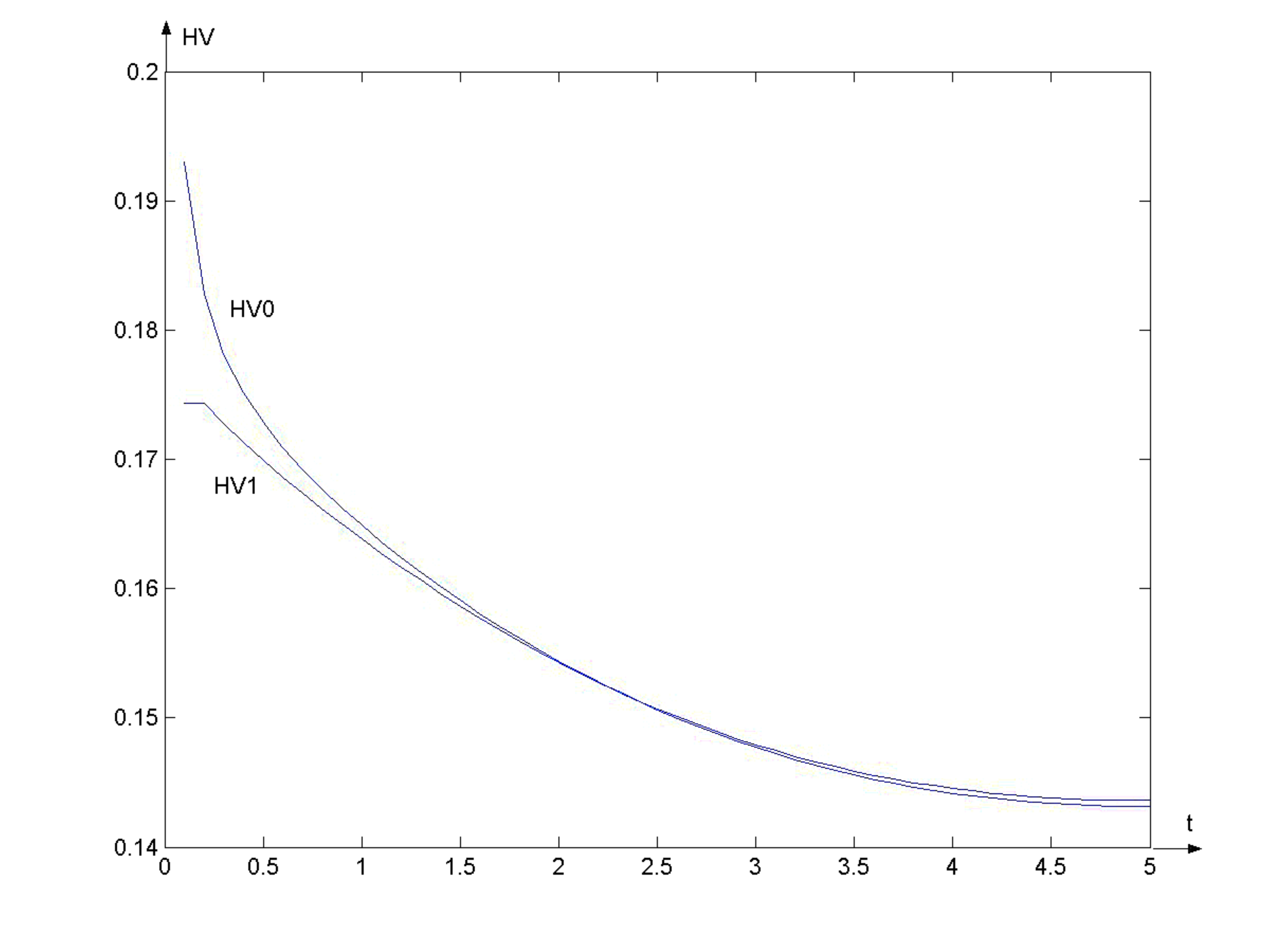}\\
 \caption{Historical volatility under the  variable velocities and jump amplitudes:
 $c_0(t)=\frac{1.2}{1+1.2t},\; h_0=\frac{-0.05}{1+1.2t};\;$ 
 $c_1=\frac{0.6}{1+0.6t},\; h_1=\frac{-0.02}{1+0.6t};\;\;$ $\lambda_0=24, \lambda_1=30$.}
 \label{fig5}
 \end{center}
\end{figure}
\begin{figure}[p]
\begin{center}
  \includegraphics[scale=0.25]{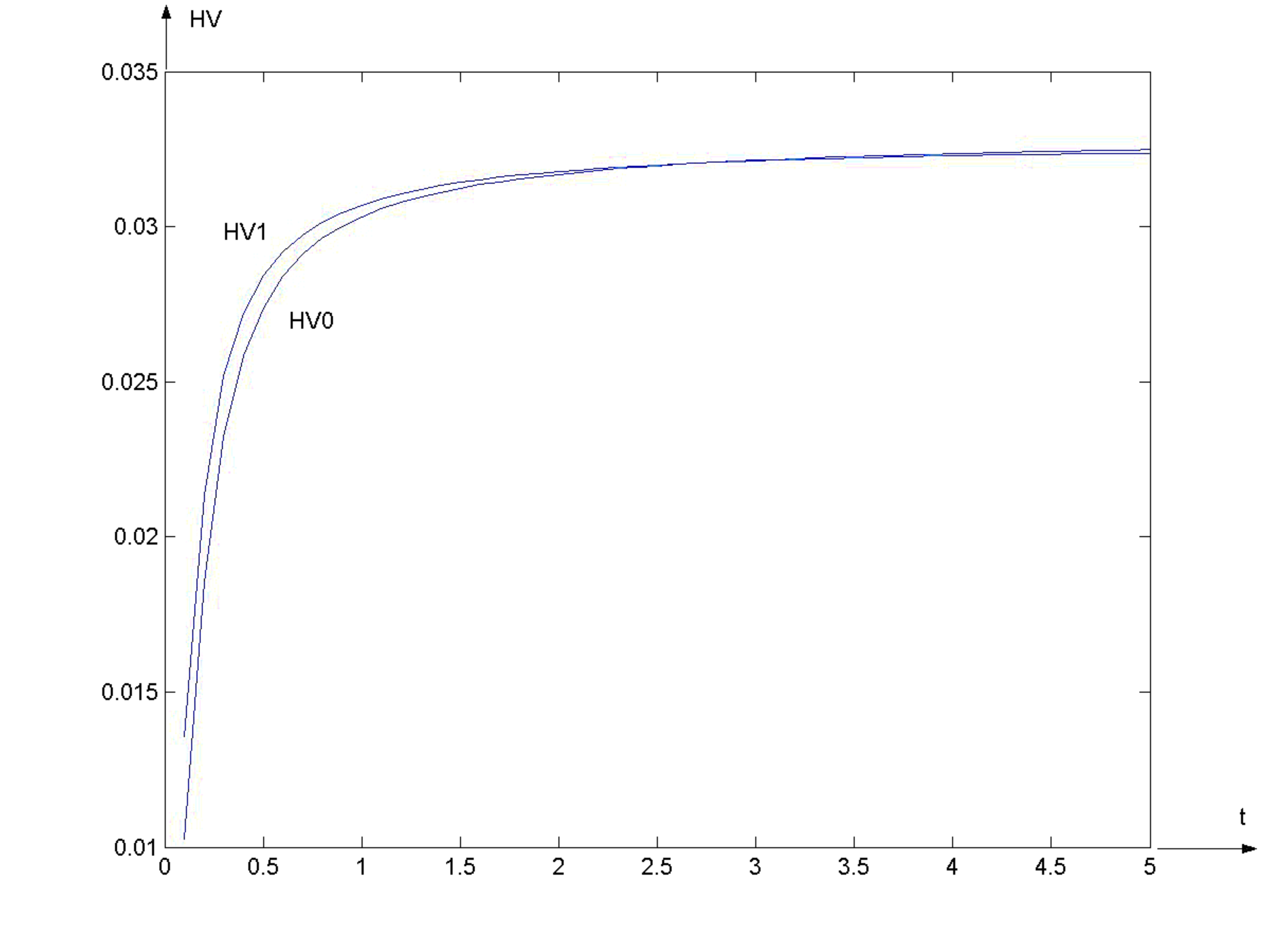}\\
 \caption{Historical volatility under the   variable velocities and jump amplitudes:
$c_0(t)=-0.5t,\; h_0=0.02t;\;$ 
 $c_1=-t,\; h_1=0.05t;\;$ $\lambda_0=25, \lambda_1=20$.}
 \label{fig6}
 \end{center}
\end{figure}

We demonstrate the memory effects related to the jump-telegraph model by means of 
the historical volatility   $\HV(\cdot)$ defined by
\begin{equation}
\label{def:HV}
\HV(t):=\sqrt{\frac{{\rm Var}\{\log S(t)\}}{t}}.
\end{equation}
For the Black-Scholes model the historical volatility is constant, 
$\HV(t)\equiv\sigma$. 

The models that capture the memory effects of the market, possess a variable historical volatility.
Consider a moving-average type model, 
which is described by the log-price
\[\begin{aligned}
\log S(t)/S(0)=&at+\sigma w(t)-\sigma\int_0^t\rmd \tau\int_{-\infty}^\tau
\lambda_0\rme^{-(\lambda_0+\lambda_1)(\tau-u)}\rmd w(u),\\
&\sigma, \; \lambda_1,\; \lambda_0+\lambda_1>0,
\end{aligned}\]
see \cite{FMM1}. This model is specially designed for the description of
exponentially decaying memory.
The historical volatility is exactly described by
\begin{equation}
\label{HV:AR}
\HV(t)=\frac{\sigma}{\lambda_0+\lambda_1}
\sqrt{\lambda_1^2+\lambda_0(2\lambda_1+\lambda_0)\varphi_\lambda(t)/t},
\end{equation}
where  $\varphi_\lambda$ is defined by \eqref{def:varphi}.
See formula (4.8) in \cite{FMM1}.

Consider the  market model, based on the stochastic exponential of jump-telegraph process
$X=X(t),\; t\in[0, U]$, see \eqref{def:S}. Surprisingly, the historical volatility of this model 
agrees with the models of a moving-average type,  see \cite{FMM1}. 
For convenience, 
we define the historical volatility in jump-telegraph  model by
$\HV_i(t):=\sqrt{\sigma_i(t)/t},\; i=0, 1$ instead of \eqref{def:HV}.
Here $\sigma_0(t)={\rm Var}\{X_0(t)\}$ and $\sigma_1(t)={\rm Var}\{X_1(t)\}$
solve system \eqref{eq:var}. The explicit formulae for $\HV_i(t)$ are rather cumbersome, 
even if  the case of constant and deterministic velocities and jumps.
Nevertheless, it is easy to compute the limits of $\HV_i(t)$ as $t\to0$ and as $t\to\infty$:
\[
\begin{aligned}
\lim\limits_{t\to 0}\HV_i(t)&=\sqrt{\lambda_i}|h_i|,\\
\lim\limits_{t\to \infty}\HV_i(t)&=
\sqrt{\frac{\lambda_0\lambda_1}{2\lambda^3}
\left[(\lambda_0B+c)^2+(\lambda_1B-c)^2\right]},
\quad i=0, 1,
\end{aligned}
\] 
see (4.5)-(4.6) in \cite{R2007Jamsa}. Here  the jump-telegraph process $X$ is defined with 
the constant velocities $c_0, c_1,\; c_0>c_1$ and with the constant jumps  $h_0,\; h_1>-1$;
$\lambda=(\lambda_0+\lambda_1)/2,\; B=(h_0+h_1)/2$ and $c=(c_0-c_1)/2$;
the subscript $i=\ep(0)$  indicates the initial market state.
\setlength{\unitlength}{1cm}

In the symmetric case, $\lambda_0=\lambda_1=\lambda$,
the historical volatility $\HV_i(t),\; t\geq0$ can be expressed by 
\begin{equation}\label{HVsymmetric}
\begin{aligned}
\HV_i(t)=&\sqrt{\frac{c^2}{\lambda}+\lambda B^2+(c+\lambda
b)^2\frac{\varphi_{2\lambda}(t)}{\lambda t}+\gamma_i\frac{\varphi_\lambda(t)}{t}
+(-1)^i2B(c+\lambda b){\rm e}^{-2\lambda t}},\\
&i=0,\; 1,
\end{aligned}
\end{equation}
where $b=(h_0-h_1)/2, \gamma_i=-2c(c/\lambda+(-1)^ih_i),\; i=0, 1,$ 
see formula (4.2) in \cite{R2007Jamsa}.

In particular, if in this symmetric case the jumps are also symmetric,
$h_0=-h_1=h,$ and $X$ is the martingale, $c+\lambda h=0$, then
$B=0,\; c+\lambda b=0$ and $\gamma_0=\gamma_1=0$. So \eqref{HVsymmetric}
gives the constant  historical volatility, $\HV_0=\HV_1\equiv c/\sqrt{\lambda}.$
In general, formula \eqref{HVsymmetric} comports with formulae 
for historical volatility of the history dependent model with memory
\eqref{HV:AR}.

Fig. \ref{fig1}-\ref{fig2}  contain the plots in the  symmetric  case.
Here $\HV_0\equiv\HV_1$. Fig. \ref{fig3}-\ref{fig4} also represent the model with constant parameters.
In these cases we use directly formula \eqref{HVsymmetric}.

Some other computations and plots of historical and implied volatilities with constant parameters $c_i, h_i, i=0,1,$
see also in \cite{R2008}.

We compute the historical volatility for the variable (deterministic) velocities and jumps
as the solution of system \eqref{eq:var} by formula \eqref{bfsigma}.
Fig. \ref{fig5} and Fig. \ref{fig6} show the result.

\end{document}